\documentclass{birkjour}
 \newtheorem{thm}{Theorem}[section]
 \newtheorem{cor}[thm]{Corollary}
 \newtheorem{lem}[thm]{Lemma}
 \newtheorem{prop}[thm]{Proposition}
 \theoremstyle{definition}
 \newtheorem{defn}[thm]{Definition}
 \theoremstyle{remark}
 \newtheorem{rem}[thm]{Remark}
 
 \numberwithin{equation}{section}

\usepackage{amsfonts,amsmath,amssymb}

\usepackage{enumerate}

\usepackage{tikz}

\usepackage{tikz-cd}

\usepackage[pdftex,colorlinks]{hyperref}

\numberwithin{equation}{section}

\newcommand{\cl}[1]{\mathcal{#1}} %\cl A

\newcommand{\bb}[1]{\mathbb{#1}}

\newcommand{\sca}[1]{\left\langle#1\right\rangle} %\sca{x,y}

\newcommand{\Lat}[1]{\mathrm{Lat}(#1)} %\Lat{A}

\newcommand{\Alg}[1]{\mathrm{Alg}(#1)} %\Alg{L} 

\newcommand{\pr}[1]{\mathrm{pr}(#1)} %\pr{A}

\begin{document}

\title[The similarity problem and hyperreflexivity of von Neumann algebras]
{The similarity problem and hyperreflexivity of von Neumann algebras}

\author{G. K. Eleftherakis}
\address{University of Patras\\Faculty of Sciences\\ Department of Mathematics\\265 00 Patras Greece}
\email{gelefth@math.upatras.gr}

\author{E. Papapetros}
\address{University of Patras\\Faculty of Sciences\\ Department of Mathematics\\265 00 Patras Greece}
\email{e.papapetros@upatras.gr}


\date{January 1, 2004}

\begin{abstract}
We say that a $C^*$-algebra $\cl A$ satisfies the similarity property ((SP)) if every bounded homomorphism $u\colon \cl A\to \cl B(H),$ where $H$ is a Hilbert space, is similar to a $*$-homomorphism and that a von Neumann algebra $\cl A$ satisfies the weak similarity property ((WSP)) if every $\mathrm{w}^*$-continuous, unital and bounded homomorphism $u\colon \cl A\to \cl B(H),$ where $H$ is a Hilbert space, is similar to a $*$-homomorphism. 
We prove that a von Neumann algebra $\cl A$ satisfies (WSP) if and only if the algebras $\cl A^{\prime}\bar \otimes \cl B(\ell^2(J))$ are hyperreflexive for all cardinals $J.$  We also introduce the hypothesis {\bf (CHH)}: Every hyperreflexive separably acting von Neumann algebra is completely hyperreflexive. We show that under {\bf (CHH)}, all $C^*$-algebras satisfy (SP). Finally, we prove that the spatial tensor product of an injective von Neumann algebra and a von Neumann algebra satisfying (WSP) also satisfies (WSP).
\end{abstract}

%%% ----------------------------------------------------------------------
\maketitle
%%% ----------------------------------------------------------------------

\section{Introduction}
One of the most famous open problems in the theory of $C^*$-algebras is  
 Kadison's similarity problem \cite{kad}: 
Given a $C^*$-algebra $ \cl A$ and a bounded homomorphism $u\colon \cl A\rightarrow \cl B(H),$ where $H$ is a Hilbert space, does there exist 
an invertible operator $S\in \cl B(H)$ such that the map $\pi (\cdot)=S^{-1}\,u(\cdot)\,S$ defines a $*$-homomorphism? In this case we say that $u$ is similar to $\pi.$ We also say 
that $\cl A$ satisfies the similarity property ((SP)) if every bounded homomorphism $u\colon \cl A\rightarrow \cl B(H),$ where $H$ is a Hilbert space, 
has the above property. 

In  \cite{haag}, Haagerup gave an affirmative answer for the case in which the map $u$ has a finite cyclic set, i.e. there exists a subset $\{\xi _1, ...,
\xi_n \}\subset H$ such that $$H=\overline{[u(A)\xi _i \mid A\in\cl A,\,\,i=1,...,n]},$$ where, in general, for a subset $\cl L$ of some vector space $X,$ $\,[\cl L]$ denotes the linear span of $\cl L.$  
In the same article Haagerup proved that  every completely bounded homomorphism from a $C^*$-algebra $\cl A$ into some $\cl B(H),$ where $H$ is a Hilbert space, is similar to a $*$-homomorphism. 

The similarity problem turns out to be equivalent to a number of questions, including the problem of the hyperreflexivity of all von Neumann algebras, as well as the derivation problem. More specifically, the following statements are equivalent :
\begin{center}
\begin{enumerate}[(1)]
  \item All unital $C^*$-algebras satisfy (SP).
  \item Every von Neumann algebra is hyperreflexive.
  \item For every $*$-representation $(H,\pi)$ of a unital $C^*$-algebra $\cl A$ and every bounded derivation $\delta\colon \cl A\to \cl B(H)$ with respect to $\pi$ there exists an operator $T\in\cl B(H)$ such that $$\delta(A)=\delta_{\pi, T}(A)=T \pi(A)-\pi(A) T,\,\,A\in\cl A.$$
\end{enumerate}
\end{center}

In \cite{ki}, Kirchberg proved that the similarity problem and the derivation problem for $C^*$-algebras are equivalent. Hyperreflexivity questions for von Neumann algebras are closely related to problems involving derivations of $C^*$-algebras. It is known from the work of Christensen \cite{chris}, that if $\cl A$ is a $C^*$-algebra in $\cl B(H),$ then every derivation from $\cl A$ into $\cl B(H)$ is inner in $\cl B(H)$ if and only the von Neumann algebra $\cl A^\prime$ is hyperreflexive. Arveson in \cite{arv} showed that all nest algebras are hyperreflexive with hyperreflexivity constant $1.$  Although many von Neumann algebras are hyperreflexive, the question of whether every von Neumann algebra is hyperreflexive remains open. For further information we refer the reader to \cite{arv,chris,dl, ki, pisier1}. 

A $\mathrm{w}^*$-closed space $\cl X$ is called completely hyperreflexive if the space  $\cl X\bar\otimes \cl B(\ell^2(\bb N))$ is hyperreflexive. It is obvious that completely hyperreflexive spaces are hyperreflexive. On the other hand, whether hyperreflexivity implies complete hyperreflexivity remains an open problem \cite{do, dl}. As we show in Section 3, the resolution of this problem, if we restrict to separably acting von Neumann algebras, yields an affirmative answer to the similarity problem. %positive solution to this problem implies solution to similarity problem in affirmative.

We introduce the following hypothesis:

\textbf{(CHH)}: Every hyperreflexive separably acting von Neumann algebra is completely hyperreflexive.  

We say that a von Neumann algebra $\cl A$ satisfies the weak similarity property ((WSP)) if  
every w*-continuous, unital and bounded homomorphism $u\colon \cl A\to \cl B(H),$ where $H$ is a Hilbert space, is similar to a $*$-homomorphism. The connection between (SP) and (WSP) is given by the following lemma. 

\begin{lem}
\label{WSP}
Let $\cl A$ be a unital $C^*$-algebra. The algebra $\cl A$ satisfies (SP) if and only if the second dual algebra $\cl A^{**}$ satisfies (WSP).
\end{lem}

\begin{proof}
    Assume that $\cl A$ satisfies (SP) and let $u\colon \cl A^{**}\to \cl B(H),$ where $H$ is a Hilbert space, be a $\rm{w}^*$-continuous, unital and bounded homomorphism. Then $u|_{\cl A}\colon \cl A\to \cl B(H)$ is a unital bounded homomorphism and since $\cl A$ satisfies (SP), $\,u|_{\cl A}$ is similar to a $*$-homomorphism, i.e. there exists an invertible operator $S\in\cl B(H)$ such that $\pi(A)=S^{-1}\,u|_{\cl A}(A)\,S,\,\,A\in\cl A\,$ is a $*$-homomorphism. The algebra $\cl A$ is $\rm{w}^*$-dense in $\cl A^{**}$ and by \cite[Theorem A.55]{pisier3} $\pi$ extends to a $\rm{w}^*$-continuous $*$-homomorphism on $\cl A^{**}.$ Therefore, $\cl A^{**}$ satisfies (WSP).
    
    Conversely, assume that $\cl A^{**}$ satisfies (WSP) and let $u\colon \cl A\to \cl B(H),$ where $H$ is a Hilbert space, be a unital bounded homomorphism. According to \cite[Theorem A.55]{pisier3} there exists a $\rm{w}^*$-continuous and unital bounded homomorphism $\tilde{u}\colon \cl A^{**}\to \cl B(H)$ extending $u.$ By the fact that $\cl A^{**}$ satisfies (WSP), there exists an invertible operator $S\in\cl B(H)$ such that $\pi(X)=S^{-1}\,\tilde{u}(X)\,S,\,\,X\in\cl A^{**}$ is a $*$-homomorphism, and thus $u$ is similar to the $*$-homomorphism $\pi|_{\cl A},$ as desired.
\end{proof}

Throughout this paper, if $\cl X$ and $\cl Y$ are dual operator spaces, then we denote by 
$\cl X\bar \otimes \cl Y$ their spatial tensor product. If $H$ is a Hilbert space with orthonormal basis of cardinality $J$ and $\cl B(H)$ is the algebra of bounded operators acting on $H,$ sometimes we identify $\cl X\bar \otimes \cl B(H)$ with the dual operator space $M_J(\cl X)$ of $J\times J$ matrices with entries in $\cl X$ which define bounded operators. If $J=\bb N$ we simply write $M_{\infty}(\cl X).$ See \cite{bm,pau,pisbook} for further details.

In Section 2, we prove that a von Neumann algebra $\cl A$ satisfies (WSP) if and only if the von Neumann algebras $\cl A^{\prime}\bar \otimes \cl B(\ell^2(J))$ are hyperreflexive for all cardinals $J.$ Here, $\cl A^{\prime}$ is the commutant of $\cl A.$ In case $\cl A\subseteq \cl B(H)$ is a von Neumann algebra and $H$ is a separable Hilbert space, $\cl A$ satisfies (WSP) if and only if the algebra $\cl A^\prime \bar \otimes \cl B(\ell^2(\bb{N}))$ is hyperreflexive. 

In Section 3, we show that under {\bf (CHH)} every separably acting von Neumann algebra satisfies (WSP). Using this result, namely all separably acting von Neumann algebras satisfy (WSP), we prove that all $C^*$-algebras satisfy (SP).

In Section 4 we prove that if $\cl A$ is an injective von Neumann algebra and $\cl B$ is a von Neumann algebra satisfying (WSP), then the algebra $\cl A\bar \otimes \cl B$ satisfies (WSP). Finally, we prove that the similarity degree $d_{*}(\cl A\bar \otimes \cl B)$ of $\cl A\bar \otimes \cl B$ is less than $2k_0+8$ where $$k_0=\sup\left\{k(\cl B^{\prime}\bar \otimes \cl B(\ell^2(J)))\mid J: \text{infinite cardinal}\right\}\in\left[1,+\infty\right).$$
Here $k(\cl B^{\prime}\bar \otimes \cl B(\ell^2(J)))$ is the hyperreflexivity constant of $\cl B^{\prime}\bar \otimes \cl B(\ell^2(J)).$

In what follows, if $(\cl X,\,||\cdot||)$ is a normed space, then by $Ball(\cl X)$ we denote the set consisting of $x\in\cl X$ such that $||x||\leq 1.$ By $\,\cl B(H,K)$ we denote the space of all linear and bounded operators from the Hilbert space $H$ to the Hilbert space $K.$  If $H=K$ we simply write $\cl B(H,H)=\cl B(H).$ Furthermore, $\,\dim H$ denotes the Hilbert dimension of a Hilbert space $H,$ that is the cardinality of an orthonormal basis of $H.$ 
If $\cl D$ is a set of operators acting on some Hilbert space $H$ then $$\cl D^\prime=\left\{X\in\cl B(H)\mid XD=DX ,\,\forall\,D\in\cl D\right\}$$ is the commutant of $\cl D$ and $\pr{\cl D}$ is the set of projections $P\colon H\to H$ such that $P\in\cl D.$ A set $\cl L\subseteq \pr{\cl B(H)}$ which contains $0,I_{H}$ and arbitrary intersections and closed spans is called a lattice. We denote by $\Alg{\cl L}$ the unital dual operator algebra $$\Alg{\cl L}=\left\{X\in\cl B(H)\mid P^{\perp} X P=0,\,\forall\,P\in \cl L\right\},$$ where $P^{\perp}=I_{H}-P$ and $I_{H}$ is the identity operator of $H.$\\
Dually, if $\cl A$ is a unital dual operator algebra acting on a Hilbert space $H,$ then the set $$\Lat{\cl A}=\left\{P\in \pr{\cl B(H)}\mid P^{\perp} A P=0,\,\forall\,A\in\cl A\right\}$$ is a lattice. For example, if $\cl A$ is a von Neumann algebra acting on the Hilbert space $H,$ then $\Lat{\cl A}=\pr{\cl A^\prime}.$

If  $\cl A$ is a $C^*$-algebra 
and $H$ is a Hilbert space then a map $\pi\colon \cl A\rightarrow \cl B(H)$ is called a $*$-representation if $\pi $ is a  $*$-homomorphism.
 In the case in which $\cl A\,$ is a von Neumann algebra and $\pi $ is a w*-continuous unital $*$-homomorphism, 
we call $\pi $ a normal representation of $\cl A$ on $H.$

Let $\cl A$ be a $C^*$-algebra and $\pi\colon \cl A\to \cl B(H)$ be a $*$-representation of $\cl A$ on the Hilbert space $H.$ A linear map $\delta\colon \cl A\to \cl B(H)$ is called a derivation with respect to $\pi$ if $$\delta(AB)=\pi(A)\delta(B)+\delta(A) \pi(B),\,\,A,B\in\cl A.$$ 
An inner derivation $\delta$ of the $C^*$-algebra $\cl A$ with respect to $\pi$ is a linear map $\delta\colon \cl A\rightarrow \cl B(H)$ of the form 
$$\delta (A)=\delta_{\pi,\,T}(A) =T\pi(A)-\pi(A) T,$$ for some $T\in \cl B(H).$ 
The derivation problem is still an open problem \cite{ki}: Can we find for every $*$-representation $\pi\colon \cl A\to \cl B(H)$ on a Hilbert space $H$ and every bounded derivation $\delta\colon \cl A \to \cl B(H)$ with respect to $\pi$ an operator $T\in\cl B(H)$ such that $\delta=\delta_{\pi,T} ?$

If $\cl A$ is a 
von Neumann algebra, an inner derivation $\delta$ with respect to some normal representation $\pi$ of $\cl A$ is called normal.

Let $\cl U$ be a subspace of $\cl B(H,K),$ where $H,\,K$ are Hilbert spaces. If $T\in \cl B(H,K)$ we set
$$d(T, \,\cl U) =\inf_{U\in\,\cl U}\|T-U\|$$ 
to be the distance from $T$ to $\cl U$. 
We also set
$$r_{\cl U}(T) =\sup_{\|\xi \|=\|\eta \|=1} \left\{|\sca{T\xi ,\eta }|\,\,\,\,\mid   \sca{U\xi, \eta }=0,\,\forall\,U\in \cl U\right\}$$ 
and it is easy to see that $r_{\cl U}(T)\leq d(T,\cl U).$ 
The reflexive hull of $\,\cl U$ is defined to be the space $$\mathrm{Ref}(\cl U)=\left\{T\in \cl B(H,K)\mid T x\in\overline{\cl U x},\,\,\forall\,x\in H\right\}.$$ Clearly, $\,\cl U\subseteq \mathrm{Ref}(\cl U)$. If $\,\cl U=\mathrm{Ref}(\cl U),$ then the space $\cl U$ is called reflexive. 
\begin{lem}
    \label{4.4}
Let $\cl A$ be a von Neumann algebra acting on the Hilbert space $H$ and $T\in\cl B(H).$ We set $$\alpha_{\cl A}(T)=\sup\left\{||P^{\perp} T P||\,\,\,\mid P\in\Lat{\cl A}\right\}.$$
Then $\alpha_{\cl A}(T)\leq r_{\cl A}(T).$
\end{lem}

\begin{proof}
    Let $P\in\Lat{\cl A}$ and $\epsilon>0.$ There exist unit vectors $\xi,\,\eta\in H$ such that \begin{equation}
        \label{in} ||P^{\perp} T P||-\epsilon<\left|\sca{P^{\perp} T P \xi,\eta}\right|=\left|\sca{T P \xi,P^{\perp} \eta}\right|.
    \end{equation}
We set $\xi_0=P \xi,\,\,\eta_0=P^{\perp} \eta.$ Then $\xi_0,\,\,\eta_0\in Ball(H)$ and for every $A\in\cl A$ it holds that $$\sca{A \xi_0,\eta_0}=\sca{A P \xi,P^{\perp} \eta}=\sca{P^{\perp} A P \xi,\eta}=\sca{P^{\perp} P A \xi,\eta}=0,$$  and thus according to (\ref{in}) we have $$||P^{\perp} T P||-\epsilon<\left|\sca{T P \xi,P^{\perp} \eta}\right|\leq r_{\cl A}(T).$$
Taking $\epsilon\to 0^{+}$ it follows that $||P^{\perp} T P||\leq r_{\cl A}(T)$ and since $P\in\Lat{\cl A}$ is arbitrary we deduce that $\alpha_{\cl A}(T)\leq r_{\cl A}(T),$ as desired.
\end{proof}

The proof of the following lemma is straightforward.
\begin{lem}
\label{denominator}
If $\,\cl U$ is reflexive and $T\notin \cl U$ then $r_{\cl U}(T)\neq 0.$
\end{lem}

\begin{defn}
\label{hyper}
If there exists $0<k<\infty$ such that 
$$d(T,\, \cl U) \leq k\,r_{\cl U}(T),\,\forall\,T\in \cl B(H,K),$$
then we say that the space $\cl U$ is hyperreflexive.
\end{defn}

Definition \ref{hyper} and Lemma \ref{denominator} lead to the following definition.

\begin{defn}
Let $\,\cl U\subseteq \cl B(H,K)$ be a reflexive subspace. We define %\ma{explain}
$$k(\cl U)=\sup_{T\not \in\,  \cl U}\frac{d(T,\, \cl U) }{r_{\cl U}(T) }$$ to be the hyperreflexivity constant of $\,\cl U.$

\end{defn}

Therefore, $\cl U$ is hyperreflexive if and only if $k(\cl U)<\infty.$ Moreover, if $\cl U$ is hyperreflexive, then $k(\cl U)\geq 1$ (since $0<r_{\cl U}(T)\leq d(T,\,\cl U)).$

We recall some basic definitions and facts that will be used in the following sections.

A linear space $\cl M\subseteq \cl B(H,K)$ is called a ternary ring of operators (TRO) if $\cl M\,\cl M^*\,\cl M\subseteq \cl M,$ i.e. $XY^*Z\in \cl M$ for all $X, Y, Z\in \cl M.$ If additionally $\cl M$ is w$^*$-closed, then $\cl M$ is called a w$^*$-TRO.

\begin{defn} \cite{el-dual,ele3}. 
\label{weakly}
  If $\,\cl X\subseteq \cl B(H_1,H_2)$ and $\cl Y\subseteq \cl B(K_1,K_2)$ are $\mathrm{w}^*$-closed operator spaces, then we say that $\cl X$ is weakly TRO equivalent to $\cl Y$ if there exist $\mathrm{w}^*$-TRO's $\,\cl M_i\subseteq \cl B(H_i,K_i),\,i=1,2$ such that $$\cl X=\overline{[\cl M_2^*\,\cl Y\,\cl M_1]}^{\mathrm{w}^*}\,\,\,\,\text{and}\,\,\,\,\,\cl Y=\overline{[\cl M_2\,\cl X\,\cl M_1^*]}^{\mathrm{w}^*}.$$  
\end{defn}

\begin{defn}
 Let $u\colon \cl X\to \cl Y$ be a linear map between operator spaces. For each $n\in\bb N$ consider the matrix amplification $$u_n\colon M_n(\cl X)\to M_n(\cl Y),\,\, u_n((X_{i,j}))=(u(X_{i,j})).$$

\end{defn}
If $\,\sup_{n} ||u_n||<\infty, $  then we say that $u$ is completely bounded with norm $$||u||_{cb}:=\sup_{n} ||u_n||.$$

A $C^*$-algebra $\cl A$ is called injective if, for given $C^*$-algebras $\cl B$ and $\cl C$ with $\cl B\subseteq \cl C$ and a completely positive map $\phi\colon \cl B\to \cl A,$ there exists a completely positive map $\hat{\phi}\colon \cl C\to \cl A$ such that $\hat{\phi}|_{\cl B}=\phi.$ 

Let $\cl A,\cl B$ be $C^*$-algebras such that $\cl A\subseteq \cl B.$ A conditional expectation is a contractive and completely positive projection $\Phi\colon \cl B\to \cl A$ such that $$\Phi(ABA')=A\Phi(B)A',\,\,\forall\,A,A'\in\cl A,\,B\in \cl B,$$ i.e. $\Phi$ is an $\cl A$-bimodule map. 
In 1957, Tomiyama proved the following theorem.

\begin{thm} \cite{Tomi} 
   Let $\cl A$ and $\cl B$ be $C^*$-algebras such that $\cl A\subseteq \cl B$ and let $\Phi\colon \cl B\to \cl A$ be a projection of $\,\cl B$ onto $\cl A.$  Then the following statements are equivalent :
   \begin{enumerate}[(i)]
       \item $\Phi$ is a conditional expectation.
       \item $\Phi$ is contractive and completely positive.
       \item $\Phi$ is contractive.
   \end{enumerate}
\end{thm}
By combining the definition of an injective von Neumann algebra and the above theorem, we obtain the following important corollary.

\begin{cor}
 \label{ce}
 For a von Neumann algebra $\cl A$ acting on a Hilbert space $H$ the following statements are equivalent :
 \begin{enumerate}[(i)]
     \item $\cl A$ is injective.
     \item There is a projection $\Phi\colon \cl B(H)\to \cl B(H)$ of $\,\cl B(H)$ onto $\cl A$ such that $||\Phi||=1.$ 
 \end{enumerate}
 Moreover, $\Phi$ is an $\cl A$-bimodule map.
\end{cor}

If $\cl A\subseteq \cl B(H)$ is a von Neumann algebra then $\cl A$ is injective if and only if $\cl A^\prime$ is injective, \cite[Corollary 3.8.6]{B-O}.

 \section{Weak Similarity property and hyperreflexivity}
 
The self-adjoint reflexive operator algebras are precisely the von Neumann algebras \cite{kl}.
In this section, using essentially Pisier's book \cite{pisier1},  we prove results connecting the weak similarity property with the notion of hyperreflexivity of von Neumann algebras. 

\begin{lem}
\label{extra}
    Let $\cl A$ and $\cl B$ be von Neumann algebras acting on the Hilbert spaces $H$ and $K,$ respectively, $\,\theta\colon \cl A^\prime \to \cl B^\prime$ be  a $*$-isomorphism and $$\cl M=\left\{T\in\cl B(H,K) \mid T A=\theta(A) T,\,\,\forall\,A\in\cl A^\prime\right\}.$$
    Then $\cl A=\overline{[\cl M^* \cl M]}^{\rm{w}^*}$ and $\,\cl B=\overline{[\cl M \cl M^*]}^{\rm{w}^*}.$
\end{lem}

\begin{proof}
    The proof of Theorem 3.3 in \cite{ele1} yields that 
    $$\cl M^* \cl B \cl M\subseteq \cl A,\,\,\cl M \cl A \cl M^*\subseteq \cl B.$$
    By Lemma 3.1 in the same article the space $\cl M$ is an essential TRO and according to \cite[Proposition 2.1]{ele1} it follows that $\cl A$ and $\cl B$ are weakly TRO equivalent (see Definition \ref{weakly}) via $\cl M,$ that is $$\cl A=\overline{[\cl M^* \cl B \cl M]}^{\rm{w^*}},\,\,\cl B=\overline{[\cl M \cl A \cl M^*]}^{\rm{w^*}}.$$
    Since  $\cl M=\overline{[\cl B \cl M]}^{\rm{w}^*}=\overline{[\cl M \cl A]}^{\rm{w}^*}$ we have
    $\cl A=\overline{[\cl M^* \cl M]}^{\rm{w}^*},\,\,\cl B=\overline{[\cl M \cl M^*]}^{\rm{w}^*}.$ \end{proof}

\begin{lem} \label{11} Let $\cl A$ and $\cl B$ be von Neumann algebras acting on the Hilbert spaces $H$ and $K,$ respectively. 
If $\,\pi\colon \cl A\rightarrow \cl B$ is an onto, $\mathrm{w}^*$-continuous and unital $*$-homomorphism and $J$ is an infinite cardinal 
 with $\dim H\leq |J|,$
then 
$$ k(\cl B^\prime \bar \otimes \cl B(\ell^2(J))) \leq k(\cl A^\prime \bar \otimes \cl B(\ell^2(J))) .$$ 
\end{lem} 
\begin{proof} We define the TRO $\cl M$ as follows: 
$$\cl M=\{M\in \cl B(H, K)\mid MA=\pi (A) M,\,\,\forall\,A\in\cl A\}.$$
We can easily prove that $\cl M\,\cl A^{\prime}\subseteq \cl M$ and $\cl M^*\,\cl B^{\prime}\subseteq \cl M^*$. Therefore, $\overline{[\cl M^*\cl M]}^{\mathrm{w}^*}$ and $\overline{[\cl M\cl M^*]}^{\mathrm{w}^*}$ are ideals of $\cl A^\prime$ and $\cl B^\prime,$ respectively. We define the map
$$\rho\colon \cl A\rightarrow \cl B(H\oplus K),\,\rho (A)=A\oplus \pi (A).$$ 
This is a $*$-isomorphism between the von Neumann algebras $\cl A$ and  $\rho  (\cl A)$. According to Lemma \ref{extra}, for  the TRO
$$\cl N=\{T\in \cl B(H,H\oplus  K)\mid T A=\rho  (A) T,\,\,\forall\,A\in\cl A\}$$ we have $$\cl A^\prime=\overline{[\cl N^* \cl N]}^{\rm{w}^*},\,\,\rho(\cl A)^\prime=\overline{[\cl N \cl N^*]}^{\rm{w}^*}.$$

However, by the proof of Theorem 3.2 in \cite{ele1} we have
$$\rho (\cl A)^\prime=\begin{pmatrix} \cl A^\prime & \cl M^* \\ \cl M  & \cl
 B^\prime\end{pmatrix}=\overline{[\cl N \cl N^*]}^{\rm{w}^*}.$$
Let $T=\begin{pmatrix}T_1\\
T_2\end{pmatrix}\in\cl B(H,H\oplus K)$, so $T_1\in\cl B(H)$ and $T_2\in\cl B(H,K)$. We have 
\begin{align*}
    T\in\cl N&\Leftrightarrow T A \xi=\rho(A) T \xi,\,\,\forall\,A\in\cl A,\,\,\forall\,\xi\in H\\&\Leftrightarrow\begin{pmatrix} T_1 A \xi\\
    T_2 A \xi\end{pmatrix}=\begin{pmatrix}A T_1 \xi\\
    \pi(A) T_2 \xi\end{pmatrix},\,\,\forall\,A\in\cl A,\,\,\forall\,\xi\in H\\&\Leftrightarrow T_1 A=A T_1,\,\,T_2 A=\pi(A) T_2,\,\,\forall\,A\in\cl A\\&\Leftrightarrow T_1\in\cl A^{\prime},\,T_2\in \cl M\\&\Leftrightarrow T\in\begin{pmatrix}\cl A^{\prime}\\
   \cl  M\end{pmatrix},
\end{align*} 
that is
\begin{equation}\label{ex1} \cl N=\begin{pmatrix} \cl A^\prime \\ \cl M \end{pmatrix}.
\end{equation}
Using (\ref{ex1}) and the fact that $\rho (\cl A)^\prime=\begin{pmatrix} \cl A^\prime & \cl M^* \\ \cl M  & \cl
 B^\prime\end{pmatrix}=\overline{[\cl N \cl N^*]}^{\rm{w}^*}$ we conclude that $$\cl B^\prime=\overline{[\cl M\cl M^*]}^{\mathrm{w}^*}.$$ 
Since $\Omega =\overline{[\cl M^*\cl M]}^{\mathrm{w}^*}$ is an ideal of $\cl A^\prime,$ there exists a projection $P$ in the centre of $\cl A$ 
such that $$\Omega =\cl A^\prime P.$$ We consider $\Omega $ as a  subspace of $\cl B(P(K)).$ Then 
$\cl B^\prime$ and $\Omega $ are weakly TRO equivalent in the sense of Definition \ref{weakly} and $J$ 
 has cardinality at least $\dim H.$ By \cite[Theorem 3.6]{ele2} we obtain 
$$k( \cl B^\prime \bar \otimes \cl B(\ell^2(J)) ) \leq k( \Omega  \bar \otimes \cl B(\ell^2(J)) ) .$$
However, $$\Omega  \bar \otimes \cl B(\ell^2(J)) =(P\otimes I)( \cl A^\prime \bar \otimes \cl B(\ell^2(J))  )(P\otimes I),$$ where $I$ is the identity operator of $\ell^2(J),$
 and thus by \cite[Lemma 1.3]{kl} we deduce that
\[k( \Omega  \bar \otimes \cl B(\ell^2(J)) ) \leq k(\cl A^\prime \bar \otimes \cl B(\ell^2(J))  ).\qedhere\]

\end{proof}

\begin{lem} \label{12} Let $\cl A$ be a von Neumann algebra such that $k_0=k(\cl A^\prime\bar \otimes \cl B(\ell^2(J))) <\infty $ 
for an infinite cardinal $J.$ For every normal representation $\pi\colon \cl A\rightarrow \cl B(H),$ where $H$ is a Hilbert space
 such that $\dim H\le |J|,$ and every $T\in\cl B(H),$ for the inner derivation 
$\delta=\delta_{\pi,\,T}$ 
 it holds that
$$\|\delta \|_{cb}\leq 2k_0 \|\delta\| .$$\end{lem}

\begin{proof}   Let $\pi\colon \cl A\to \cl B(H)$ be a normal representation of $\cl A$ on a Hilbert space $H$ such that $\dim H\le |J|$ and $\,T\in\cl B(H).$ By Lemma \ref{11} we have 
$k( \pi (\cl A)^\prime \bar \otimes \cl B(\ell^2(J))) \leq k_0.$ Therefore, $k(\pi (\cl A)^\prime )\leq k_0$ and 

\begin{equation}\label{xx} d(T,\,\pi (\cl A)^\prime )\leq k_0 \,r_{\pi (\cl A)^\prime}(T) %\quad\text{for all }\; T\in \cl B(H) .
\end{equation}

Let $n\in\bb N$ and $(A_{i,j})\in Ball(M_{n}(\cl A)).$ For all $X\in \pi(\cl A)^{\prime}$ we have 
 \begin{align*} \|\delta_n((A_{i,j}))\|_{M_{n}(\cl B(H))}&=\|(\delta(A_{i,j}) )\| _{M_{n}(\cl B(H))} \\&= \|(T\pi (A_{i,j})-\pi (A_{i,j})T )\|_{M_{n}(\cl B(H))} \\&=
 \|((T-X)\pi (A_{i,j}))-(\pi (A_{i,j})(T-X) )\|_{M_{n}(\cl B(H))} \\&\leq 2\|T-X\|\|(A_{i,j})\|_{M_{n}(\cl A)}.
\end{align*}
We deduce that $\|\delta\|_{cb}\leq 2\,\|T-X\|,\,\forall\,X\in\pi(\cl A)^{\prime},$ which implies that \begin{equation} \label{2.3}
    \|\delta\|_{cb}\leq 2\,d(T,\,\pi(\cl A)^{\prime}).
\end{equation}
On the other hand, by \cite[Lemma 1.2]{ele2} it holds that
\begin{align*} r_{\pi (\cl A)^\prime}(T)&=\sup\big\{\|P^\bot TP\| \mid  P\in \pr{\pi (\cl A)}\big\}\\&\leq 
\sup\big\{\|TP-PT\|\mid P\in \pr{\pi (\cl A)}\big\}\\&\leq  \sup\big\{\|T\pi (A)-\pi (A)T\| \mid A\in Ball(\cl A)\big\}\\&=\|\delta \|.
\end{align*}
The last inequality, together with (\ref{xx}) and  (\ref{2.3}) implies  that \[\|\delta \|_{cb}\leq 2k_0 \|\delta\|. \qedhere\] 

\end{proof}

\begin{lem}\label{13} Let $\cl A$ be a von Neumann algebra such that $ k(\cl A^\prime\bar \otimes \cl B(\ell^2(J))) <\infty $ 
for an infinite cardinal $J.$ Then there exists a nondecreasing map $f\colon \mathbb{R}_+ \rightarrow \mathbb{R}_+$ 
such that for every normal representation $\pi\colon \cl A\rightarrow \cl B(H)$ such that $\dim H\le |J|$
%where the cardinal of the orthonormal basis of $H$ is less or equal to $I,$ 
and every invertible operator 
$S\in \cl B(H)$ it holds that $\|u_S\|_{cb}\leq f(\|u_S\|),$ where $$u_S(\cdot)=S^{-1}\,\pi (\cdot)\,S.$$
\end{lem}

 The proof of this lemma, which also uses Lemma \ref{12}, is the same as that of the implication  (v)$\Rightarrow$(iii) in \cite[Theorem 7.21]{pisier1}, so we omit it.

\begin{lem} \label{14} Let $\cl A$ be a von Neumann algebra such that $ k(\cl A^\prime\bar \otimes \cl B(\ell^2(J))) <\infty $ 
for an infinite cardinal $J.$ If $\,H$ is a Hilbert space, $u\colon \cl A\rightarrow \cl B(H)$ is a $\mathrm{w}^*$-continuous, unital and bounded homomorphism 
and either $\dim H\le |J|$ or $u(\cl A)$ is a $\mathrm{w}^*$-separable space, 
then $u$ is similar to a $*$-homomorphism of $\cl A.$
\end{lem}
\begin{proof} Let $(\xi_f)_{f\in F}$ be a finite subset of $H.$ We set  
$$H_F=\overline{[u(\cl A)\xi _f \mid f\in F]}^{\|\cdot\|}\subseteq H.$$ Observe that if $\dim H\le |J|,$ 
then the same holds for $H_F.$  Additionally, if $u(\cl A)$ is a w*-separable space, then $H_F$ is separable, 
 and thus $\dim H_F\le |J|$. Let
 $$u_F\colon \cl A\rightarrow \cl B(H_F),\,\,u_F(A)=u(A)|_{H_F}.$$ This map is a w*-continuous unital bounded homomorphism. 
Since $u_F$ admits a finite cyclic set, it is completely bounded \cite[Proposition 1.9]{haag}, so there exists an invertible operator $S\in \cl B(H_F)$ 
such that the map $$\pi _F(\cdot)=S^{-1}\,u_F(\cdot)\,S$$ is a normal representation of $\cl A.$ By Lemma \ref{13} 
 there exists a nondecreasing map $f\colon \mathbb{R}_+ \rightarrow \mathbb{R}_+$ such that  
$$\|u_F\|_{cb}\leq f(\|u_F\|)\leq f(\|u\|).$$
Considering for each $n\in\bb{N}$ the amplification  $u_n\colon M_{n}(\cl A)\rightarrow M_{n}(\cl B(H))$ we obtain  
$$\|u_n\|\leq \sup_{F\subseteq H,\,|F|=n}\|u_F\|_{cb} \leq f(\|u\|) .$$ 
Therefore, $u$ is completely bounded and \[\|u\|_{cb}\leq f(\|u\|) .\qedhere\]

\end{proof}

\begin{lem}\label{100} Let $\cl A$ be a von Neumann algebra acting on a Hilbert space $K$ 
and let $J$ be an infinite cardinal such that $\dim K\leq |J|$ and such that 
 every $\mathrm{w}^*$-continuous, unital and bounded homomorphism $u\colon \cl A\rightarrow \cl B(H),$ where $H$ is a Hilbert space 
with  $\dim H\leq |J|,$ is similar to a $*$-homomorphism of $\cl A.$ Then $$k ( \cl A^\prime\bar \otimes \cl B( \ell^2(J) )) <\infty 
.$$
\end{lem}

\begin{proof} Let $\cl M=\cl A\bar \otimes \bb C I_{ \ell^2(J) },$ where $I_{\ell^2(J)}$ is the identity operator of $\ell^2(J).$ According to the notes and remarks of \cite[Chapter 4]{pisier1}, it suffices to prove that 
if $\,\delta \colon \cl M\rightarrow \cl B( K\otimes \ell^2(J) )$ is a bounded derivation, then $\delta $ 
is completely bounded. We define the map 

$$u\colon \cl A \rightarrow \cl B ( (K \otimes  \ell^2(J)) \oplus (K \otimes  \ell^2(J))  ),\,\,u(A)=\begin{pmatrix}
    A\otimes I_{ \ell^2(J)}  & \delta (A\otimes 
I_{ \ell^2(J)} )\\ 0 & A\otimes I_{ \ell^2(J) }
\end{pmatrix}.$$
This map is a unital bounded homomorphism. By \cite[Theorem 4]{ring},  
$\,\delta $ is w*-continuous, and hence $u$ is w*-continuous. Since $$\dim \left( \left (K\otimes \ell^2(J) )\oplus (K \otimes  \ell^2(J)\right) \right)\leq 
2\dim K\cdot|J|=|J|,$$ by our assumptions $u$ is similar to a $*$-homomorphism of $\cl A.$ Theorem 1.10 in \cite{haag} implies that $u$ is completely bounded, so 
$\delta $ 
is completely bounded.
\end{proof}

\begin{lem}\label{102} Let $J_1$ and $J_2$ be cardinals such that $|J_1|\leq |J_2|.$ If $\cl X\subseteq \cl B(H)$ is a $\mathrm{w}^*$-closed operator space, where $H$ is a Hilbert space, then $$k(\cl X  \bar \otimes \cl B(\ell^2(J_1)))\leq k(\cl X  \bar \otimes \cl B(\ell^2(J_2))).$$
\end{lem}

\begin{proof}
We assume that $k(\cl X\bar \otimes \cl B(\ell^2(J_2)))<\infty.$ Since $|J_1|\leq |J_2|$ we may consider $\ell^2(J_1)$ as a norm closed subspace of $\ell^2(J_2).$ Let $P\colon \ell^2(J_2)\to \ell^2(J_1)$ be the projection onto $\ell^2(J_1).$ We have $$\cl X\bar \otimes \cl B(\ell^2(J_1))=(I_{H}\otimes P)\left(\cl X\bar \otimes \cl B(\ell^2(J_2))\right)(I_{H}\otimes P),$$ where $I_{H}$ is the identity operator of $H.$ By \cite[Lemma 1.3]{kl} we conclude that \[k(\cl X  \bar \otimes \cl B(\ell^2(J_1)))\leq k(\cl X  \bar \otimes \cl B(\ell^2(J_2))).\qedhere\]
\end{proof}

\begin{thm}\label{101} \label{basic} Let $\cl A$ be a von Neumann algebra. The algebra $\cl A$ satisfies (WSP)   
 if and only if the algebras  
$ \cl A^\prime\bar \otimes \cl B( \ell^2(J) ) $ are hyperreflexive for all cardinals $J.$ 
\end{thm}
\begin{proof} By combining Lemmas \ref{14} and \ref{100}, we obtain the following:
The algebra $\cl A$ satisfies (WSP)   
 if and only if there exists a cardinal $J_0$ such that  
$ \cl A^\prime\bar \otimes \cl B( \ell^2(J) ) $ are hyperreflexive for all cardinals $J$ with $|J|\ge |J_0|.$ But if $ \cl A^\prime\bar \otimes \cl B( \ell^2(J_0) )$ is hyperreflexive, it follows from Lemma \ref{102} that the algebras $ \cl A^\prime\bar \otimes \cl B( \ell^2(J) )$ are hyperreflexive for all cardinals $J$ such that $|J|\le |J_0|.$
\end{proof}

\begin{rem}\label{ttt} If $\cl A$ satisfies (WSP) then $$\sup_J k(\cl A^\prime\bar \otimes \cl B( \ell^2(J) ))<\infty.$$
Indeed, there exists a sequence of cardinals $(J_n)_{n\in \bb N}$ such that 
$$\lim_n k(\cl A^\prime\bar \otimes \cl B( \ell^2(J_n)))=\sup_J k(\cl A^\prime\bar \otimes \cl B( \ell^2(J) )).$$  
If we define $$J_0=\bigcup_{n}J_n,$$ then 
$$\sup_J k(\cl A^\prime\bar \otimes \cl B( \ell^2(J) ))=k(\cl A^\prime\bar \otimes \cl B( \ell^2(J_0)))< \infty.$$
\end{rem}

\begin{cor}\label{104} If  $\cl A\subseteq \cl B(H)$ is a von Neumann algebra and $H$ is a separable Hilbert space, then 
$\cl A$ satisfies (WSP) if and only if $k( \cl A^\prime\bar \otimes \cl B( \ell^2(\bb N)))<\infty .$ 
\end{cor}
\begin{proof} If $\cl A$ satisfies (WSP) then by Theorem \ref{101} we have $$k(\cl A^\prime\bar \otimes \cl B( \ell^2(\bb N)))<\infty.$$ 
Conversely, if  $k(\cl A^\prime\bar \otimes \cl B( \ell^2(\bb N )))<\infty $ and    
$u\colon \cl A\rightarrow \cl B(K),$ where $K$ is a Hilbert space, is a w*-continuous, unital and bounded homomorphism, then    
the space $u(\cl A)$ is w*-separable. Indeed, let $D$ be a countable subset of $\cl A$ such that $\cl A=\overline{[D]}^{\mathrm{w}^*}.$ Since $u$ is w*-continuous, we have that $u(\cl A)\subseteq \overline{[u(D)]}^{\mathrm{w}^*}.$ The conclusion follows from Lemma \ref{14}.
\end{proof}

\begin{prop}\label{berc} Let $\cl A$ be a von Neumann algebra and let $v,\,w\in \cl A$ be isometries with orthogonal ranges. Then $\cl A$ satisfies (WSP). 
\end{prop}

\begin{proof} Let $J$ be an infinite cardinal. The commutant of the space $\cl A^\prime\bar \otimes \cl B( \ell^2(J ))$ is the algebra $\cl A\bar \otimes \bb C I_{\ell^2(J)},$ where $I_{\ell^2(J)}$ is the identity operator of $\ell^2(J).$ This algebra contains the isometries $v\otimes I_{\ell^2(J)}$ and $w\otimes I_{\ell^2(J)}.$ By \cite[Corollary 4.4]{hari} the space $\cl A^\prime\bar \otimes \cl B( \ell^2(J ))$ is hyperreflexive, so Theorem \ref{101} yields that $\cl A$ satisfies (WSP).
\end{proof}

\begin{rem} A different proof of the above corollary is given in \cite[Proposition 7.13]{pisier1}.
\end{rem}

\begin{rem} The von Neumann algebras of type $II_\infty$ or type $III$ are  properly infinite. The properly infinite von Neumann algebras have  isometries with orthogonal ranges, see \cite[Theorem 1.36]{tak}. By Proposition \ref{berc} these algebras satisfy (WSP). Since type I algebras also satisfy (WSP), the similarity problem has a positive solution if
and only if the type $II_1$ von Neumann algebras satisfy (WSP).
\end{rem}

\section{(CHH) implies that all C*-algebras satisfy (SP)}

\begin{lem}\label{SOS} Assume that {\bf (CHH)} holds. If $\cl A$ is  a von Neuamnn algebra with a cyclic vector acting on a separable infinite dimensional Hilbert space $H,$ then $\cl A$ satisfies (WSP).

\end{lem}

\begin{proof}
    The cyclic vector for the algebra $\cl A$ is a seperating vector for the commutant $\cl A^\prime.$  By \cite[Remark 3.5]{kl} it follows that $\cl A^\prime$ is hyperreflexive. Under our hypothesis, the algebra 
     $\cl A^\prime \bar \otimes \cl B(\ell^2(\bb N))$ is also hyperreflexive and according to Corollary \ref{104} the algebra $\cl A$ satisfies (WSP).
\end{proof}

\begin{thm}\label{SOL} Assume that {\bf (CHH)} holds. If $\,\cl A$ is a separably acting von Neumann algebra, then $\cl A$ satisfies (WSP).\end{thm}
\begin{proof} 
 We consider the direct integral decompositions
 $$\cl A^\prime = \int_\Gamma^{\oplus}\cl A_\gamma^\prime\, d\mu(\gamma),\,\,\,\cl A=\int_\Gamma^{\oplus}\cl A_\gamma\, d\mu(\gamma),$$
 where $(\Gamma,\mu)$ is a standard measure space and $\,\cl A_\gamma^\prime,\, \cl A_\gamma $ are factors for $\mu$-a.a.$\,\gamma\in\Gamma.$ Let $$H=\int_\Gamma^\oplus H_\gamma\, d\mu (\gamma)$$ be the Hilbert space on which $\cl A$ acts.
 
 \textbf {Claim:} There exists a non-zero vector $\xi \in H$  such that, if $$K=\overline{\cl A\,\xi}=\int_{\Gamma}^{\oplus}\overline{\cl A_\gamma\,\xi(\gamma)}\,d \mu(\gamma),\,
\,\,K_\gamma=\overline{\cl A_\gamma\,\xi(\gamma)}$$ and  $E=\{\gamma \in \Gamma\mid  K_\gamma=\left\{0\right\}\},$
then $\mu(E)=0.$

{\bf Proof of the Claim:} 
We assert that if $F$ is a Borel subset of $\Gamma$ such that $0<\mu(F)<\infty,$ then there exists a vector $\eta_0\in H$ such that 
$$\{\gamma\in \Gamma\mid \eta_0(\gamma)\neq 0\}=F\,\,\,\text{$\mu$-a.e.}$$  
Indeed, by Zorn's Lemma we choose a maximal family of non-zero vectors of $H$ with disjoint supports, $G=(\eta_k)_{k\in\bb N}$ such that 
$$F_0=\{\gamma\in \Gamma\mid \exists\,k\in\bb N,\,\eta_k(\gamma)\neq 0\}\subseteq F.$$
If $\,\mu(F\setminus F_0)>0,$ then there exists $\eta \in H$ such that 
$\eta=\eta\,\chi_{F\setminus F_0}\in H\setminus \{0\}.$ However, this fact contradicts the maximality of  $G,$ and thus $\mu(F\setminus F_0)=0.$
We set $$\eta_0=\sum_{k=1}^{\infty}\frac{1}{k\|\eta_k\|_2}\,\eta_k\in H$$ and we have $\left\{\gamma\in\Gamma\mid \eta_0(\gamma)\neq 0\right\}=F\,\,\,\text{$\mu$-a.e.,}$ as desired. 

We now choose an arbitrary $\xi \in H\setminus \left\{0\right\}$ and we define $K$ and $E$ as in the claim. If  $\mu(E)=0,$ we are ok. If $\mu(E)>0,$ then there exists an increasing  sequence of sets $(L_k)_{k\in\bb N}$ 
such that $$0<\mu(L_k)<\infty\,\,\,\text{and}\,\,\,\mu\left(E\setminus \bigcup_{k=1}^{\infty}L_k\right)=0$$ and there exist vectors $(\eta_k)_{k\in\bb N}\subseteq H$ 
such that $\|\eta_k\|\le 1$ and
$$L_k=\{\gamma \in \Gamma\mid \eta_k(\gamma)\neq 0\}\subseteq E,\,\,\forall\,k\in\bb N.$$
For each $k\in\bb N$ we consider the sets $$M_k=L_k\setminus \bigcup_{i=1}^{k-1}L_i$$ and the vectors $$\zeta_k=\frac{1}{2^k}\,\chi_{M_k}\,\eta_k,\,\,\,\zeta=\sum_{k=1}^\infty \zeta_k.$$ We define the space $$\cl M=\overline{\cl A(\xi+\zeta)}$$ and we are going to prove that $\mu(\left\{\gamma\in\Gamma\mid \cl M_\gamma=\left\{0\right\}\right\})=0.$ To this end, let $\gamma\in\Gamma.$
It holds that $$\overline{\cl A_\gamma\,\xi(\gamma)}\perp \overline{\cl A_{\gamma}\,\zeta(\gamma)}.$$
Indeed, if $X,\,Y\in\cl A_{\gamma},$ then \begin{align*}
    \sca{X(\xi(\gamma)), Y(\zeta(\gamma))}&=\sca{X(\xi(\gamma)), Y(\zeta(\gamma)\,\chi_{E}(\gamma))}\\&=\sca{X(\xi(\gamma)\,\chi_{E}(\gamma)), Y(\zeta(\gamma))}\\&=
\sca{0, Y(\zeta(\gamma))}\\&=0
\end{align*}

It follows that $\overline{\cl A\,\xi}\perp \overline{\cl A\,\zeta}.$ Therefore,
\begin{align*}
   \mu(\{\gamma \in \Gamma\mid  \cl M_\gamma=\left\{0\right\}\})&=\mu(\{\gamma \in \Gamma \mid K_{\gamma}=\left\{0\right\}\,\,\,\text{and}\,\,\, \zeta(\gamma)=0\})\\&=
\mu\left(E\setminus \bigcup_{k=1}^{\infty}M_k\right)=0, 
\end{align*} 
as desired.

By the above claim, we may fix a non-zero vector $\xi \in H$  
 such that if $$K=\overline{\cl A\,\xi}=\int_{\Gamma}^{\oplus}\overline{\cl A_\gamma\,\xi(\gamma)}\,d\mu(\gamma),$$ then $$\mu(\{\gamma \in \Gamma\mid K_\gamma=\left\{0\right\}\})=0.$$
 We denote by $P$ the projection onto $K$ and clearly $P\in\cl A^\prime.$ 
 We consider the $*$-homomorphism 
 $$\cl A\to \cl A P,\,\, X\mapsto X P.$$
 This map induces $*$-homomorphisms $$\cl A_\gamma\to \cl A_\gamma P_\gamma,\,\,X\mapsto X P_\gamma.$$
Since $\cl A_\gamma$ are factors and $P_\gamma \neq 0$ for $\mu$-a.a.$\,\gamma\in\Gamma,$ the above maps are $*$-isomomorphisms for $\mu$-a.a.$\,\gamma\in\Gamma,$ so
$$\cl A_\gamma \cong \cl A_\gamma P_\gamma $$  for $\mu$-a.a.$\,\gamma\in\Gamma.$ 
Therefore $$\cl A\cong \int_{\Gamma}^{\oplus}\cl A_\gamma P_\gamma\,d\mu(\gamma)=\cl A P.$$
  According to \cite[Theorem 4.6]{ele2} we have $k(M_{\infty}(\cl A^\prime))=k(M_{\infty}(P \cl A^\prime P)).$ 
  
  However, the vector $\xi$ is a separating vector for the algebra $P \cl A^\prime P$ and by \cite[Remark 3.5]{kl} it holds that $k(P \cl A^\prime P)\leq 48<\infty.$ Under our hypothesis it follows that $k(M_{\infty}(P \cl A^\prime P))<\infty,$ and thus $$k(M_{\infty}(\cl A^\prime))=k(M_{\infty}(P \cl A^\prime P))<\infty.$$ By Corollary \ref{104} the algebra $\cl A$ satisfies (WSP).
\end{proof}

\begin{lem}\label{an} Let $\cl X$ be a dual Banach space, let $\cl X_0,\,\cl Y,\,\cl Z$ be Banach spaces such that 
$\cl X=\overline{\cl X_0}^{\mathrm{w}^*}$ and let $\phi\colon \cl X\to \cl Y^{**}$ be a $\mathrm{w}^*$-continuous onto isometry such that $\phi(\cl X_0)=\cl Y.$ If $u\colon \cl X_0\to \cl Z^*$ is a bounded linear map, then there exists a $\mathrm{w}^*$-continuous bounded linear map $\tilde u\colon \cl X\to \cl Z^*$ such that $\tilde{u}|_{\cl X_0}=u$ and $\|u\|=\|\tilde u\|.$
\end{lem}

\begin{proof} We consider the map $$u\circ \phi^{-1}\colon \cl Y\to \cl Z^*.$$
By \cite[A.2.2]{bm}, this map extends to a $\mathrm{w}^*$-continuous map 
$v\colon \cl Y^{**}\to \cl Z^*$ with $\|v\|=\|u\circ \phi^{-1}\|=\|u\|.$ We set 
$\tilde u=v\circ \phi\colon \cl X\to \cl Z^*$ and it holds that $\|u\|=\|\tilde u\|.$ Furthermore, for all $X\in \cl X_0$ we have 
\[\tilde u(X)=v\circ \phi(X)=u\circ \phi^{-1}(\phi(X))=u(X).\qedhere\]

\end{proof}

\begin{lem}\label{RR} Let $\cl A$ be a $C^*$-algebra such that $\cl A\subseteq \cl A^{**}\subseteq \cl B(H)$ and let $P\in \cl A^\prime$ be a projection. Then there exists a $C^*$-algebra $\cl D$ 
and a $*$-isomorphism  $\alpha\colon \cl A P\to \cl D$ which extends to a $*$-isomorphism from 
$\overline{\cl A P}^{\mathrm{w}^*}$ to $\cl D^{**}.$ 
\end{lem}

\begin{proof} Since $P\in\cl A^\prime,$ the map $$\cl A\to \cl A P\colon  X\mapsto X P$$ is a bounded $*$-homomorphism, and thus it can be extended to a $\mathrm{w}^*$-continuous map from $\cl A^{**}$ onto $\overline{\cl A P}^{\mathrm{w}^*}.$ Therefore, there exists a central projection $Q\in \cl A^{**}$ such that $\cl A^{**}Q\cong \overline{\cl A P}^{\mathrm{w}^*}.$ The last isomorphism induces a $*$-isomorphism $$\lambda\colon \cl A P\to \cl A Q, \,\,\,X P\mapsto X Q$$ which extends to a 
$\mathrm{w}^*$-continuous isomorphism from $\overline{\cl A P}^{\mathrm{w}^*}$ to $\cl A^{**}Q.$
We consider the canonical embedding $$\iota\colon \cl A Q\to (\cl A Q)^{**}$$
as well as the map 
$$\tau\colon \cl A^{**}\to (\cl A Q)^{**}$$ such that $\tau(X)=\iota(X Q),\,\,\forall\,X\in \cl A.$
Moreover, we define the  $\mathrm{w}^*$-continuous onto map $$\sigma\colon (\cl A Q)^{**}\to \overline{\cl A Q}^{\mathrm{w}^*}, \,\,\sigma(\iota(X Q))=X Q,\,\,\forall\,X\in \cl A.$$
 Let $X\in  \cl A^{**}$ such that $\tau(X)=0.$ There exists a net 
 $(X_j)_{j\in J}$ of elements in $\cl A$ such that $X_j\stackrel{\mathrm{w}^*}{\to} X,$ and thus
 $$\tau(X_j)\stackrel{\mathrm{w}^*}{\to} \tau(X)\Rightarrow \iota(X_j Q)\stackrel{\mathrm{w}^*}{\to} 0,$$ which implies that 
 $$ \sigma(\iota(X_j Q))\stackrel{\mathrm{w}^*}{\to} 0\Rightarrow X_j Q\stackrel{\mathrm{w}^*}{\to} 0\Rightarrow X Q=0\Rightarrow X=X\,Q^\bot.$$
 Thus, $\ker \tau \subseteq \cl A^{**}Q^\bot.$

 On the other hand, if 
 $X=X Q^\bot\in \cl A^{**}\,Q^\bot,$  then there exists a net   
$(X_j)_{j\in J}$ of elements in $\cl A$ such that  $X_j\stackrel{\mathrm{w}^*}{\to} X.$ We have 
$$\tau(X_j)\stackrel{\mathrm{w}^*}{\to} \tau(X)\Rightarrow \iota(X_j Q)\stackrel{\mathrm{w^*}}{\to} \tau(X)\Rightarrow \sigma(\iota(X_j Q))\stackrel{\mathrm{w}^*}{\to} \sigma(\tau(X)).$$
Therefore, $$X_j Q\stackrel{\mathrm{w}^*}{\to} \sigma(\tau(X))\Rightarrow \sigma(\tau(X))=0,$$ that is,
$$\tau(X)\in \ker \sigma \subseteq (\cl A Q)^{**},$$
so there exists a net $(Y_k)_{k\in K}$ of elements in $\cl A$ 
such that 
$$\tau(X)=\lim_{k}\iota(Y_k Q)\Rightarrow 0=\sigma(\tau(X))=\lim_{k} \sigma(\iota(Y_k Q))=\lim_{k}Y_k Q,$$ where the above limits are on the $\mathrm{w}^*$-topology.
However, $$0=\lim_{k}\tau(Y_k\,Q)=\lim_{k}\iota(Y_k\,Q)$$ on the $\mathrm{w}^*$-topology, and thus
$$\tau(X)=0\Rightarrow X\in \ker \tau.$$
We conclude that $$\ker \tau =\cl A^{**}Q^\bot,$$ so 
$$(\cl AQ)^{**}\cong \cl A^{**}/\ker \tau \cong \cl A^{**}Q.$$
We proved that there exists a $*$-isomorphism 
$$\iota: \cl A Q\to \iota(\cl A Q)$$ 
which extends to a $*$-isomorphism $$\tau: \cl A^{**}Q\to (\cl A Q)^{**},$$
 satisfying $$\tau(X Q)=\iota(X Q),\,\,\forall\,X\in\cl A.$$
Thus, there exists a  $*$-isomorphism 
$$\alpha=\tau\circ \lambda\colon \cl A P\to \iota(\cl A Q)$$
 which extends to a $*$-isomorphism from $\overline{\cl A P}^{\mathrm{w}^*}$ to $(\cl A Q)^{**}.$ 
  We set $\cl D=\cl A Q$ and the assertion follows.
  \end{proof}

In the following Lemma \ref{RRR} and Theorem \ref{sp} we assume that {\bf (CHH)} is true and we show that all $C^*$-algebras satisfy (SP). 

\begin{lem}\label{RRR} Assume that {\bf (CHH)} holds. Let $\cl A$ be a separable $C^*$-algebra such that  $\cl A\subseteq \cl A^{**}\subseteq  \cl B(H)$ and $I_H\in \cl A.$ If $\,u\colon \cl A\to \cl B(K)$ is a unital bounded homomorphism, then $u$ is completely bounded.
 \end{lem}
 \begin{proof} Since $M_n(\cl A)$ is also a separable algebra, there exists a countable set $$\left\{(A_{i,j}^{n,m})_{i,j}\in M_{n}(\cl A)\mid A_{i,j}^{n,m}\in\cl A,\,m\in\bb N\right\}$$ which is norm-dense in $M_n(\cl A).$ Moreover, there exists a sequence of unit vectors $$\xi_k^{n,m}=(\xi_k^{n,m}(s))_{1\le s\le n},\,\,\xi_k^{n,m}(s)\in H,\,\,m,k \in \bb N$$
 such that 
 $$\|(A_{i,j}^{n,m})\|=\sup_k\|(A_{i,j}^{n,m})\xi_k^{n,m}\|.$$
 Similarly, there exists a sequence of unit vectors $$\eta_k^{n,m}=(\eta_k^{n,m}(s))_{1\le s\le n},\,\,\eta_k^{n,m}(s)\in K,\,\,m,k \in \bb N$$ such that $$\|(u(A_{i,j}^{n,m}))\|=\sup_k\|(u(A_{i,j}^{n,m}))\eta_k^{n,m}\|,\,\,m\in \bb N.$$
 We consider the separable Hilbert space 
 $$L=\overline{\left[\begin{pmatrix}A_{i,j}^{n,m} & 0\\ 0 & u(A_{i,j}^{n,m}) \end{pmatrix}\,\begin{pmatrix}\xi_{k}^{t,r}(s) \\ \eta_{k}^{t,l}(s) \end{pmatrix}\mid n,\,m,\,t,\,l,\,r \in \bb N,\, 1\le s\le n\right]}.$$
 Clearly, $$L=\overline{\left[\begin{pmatrix} A & 0\\ 0 & u(A) \end{pmatrix}\,\begin{pmatrix} \xi_{k}^{n,r}(s) \\ \eta_{k}^{n,l}(s) \end{pmatrix} \mid  A\in \cl A,\,n,\,l,\,r\in \bb N,\, 1\le s\le n\right]}.$$
 
 We set $P$ for the projection onto $$\overline{[A\,\xi_{k}^{n,r}(s)\in H\mid A\in\cl A,\,n,\,k,\,r\in \bb N,\,1\leq s\leq n]}\subseteq H$$ and $Q$ for the projection onto $$\overline{[u(A)\,\eta_{k}^{n,l}(s)\in H\mid A\in\cl A, \,n,\,k,\,l\in \bb N,\,1\leq s\leq n]}\subseteq K.$$ Then, $L\subseteq P(H)\oplus Q(K)$ and we observe that  
 $$A\,P(h)\in P(H),\,\,\forall\,A\in\cl A,\,h\in H,$$ which implies that $P\in\cl A^\prime.$ Furthermore, it is obvious that
 $$\{\xi_{k}^{n,r}(s)\mid 1\leq s\leq n\}_{\,n, \,r,\,k\in\bb N}\subseteq P(H),\,\,\{\eta_{k}^{n,l}(s)\mid 1\leq s\leq n\}_{\,n, k,\,l\in\bb N}\subseteq Q(K).$$
 The Hilbert spaces $L,\,P(H)$ and $Q(K)$ are all separable. On the other hand, 
 \begin{align*}
    \|(A_{i,j}^{n,m})\|&=\sup_k\|(A_{i,j}^{n,m})\xi_k^{n,m}\|\\&=\sup_k\|(A_{i,j}^{n,m}P)\xi_k^{n,m}\|\\&\le \|(A_{i,j}^{n,m}P)\|\\&\le \|(A_{i,j}^{n,m})\|,
 \end{align*}
 and thus $$\|(A_{i,j}^{n,m})\|=\|(A_{i,j}^{n,m}P)\|,\,\,m\in \bb N,\,\,1\leq i,\,j\leq n,\,\,n\in \bb N.$$
 Similarly, it holds that
 $$\|(u(A_{i,j}^{n,m}))\|=\|(u(A_{i,j}^{n,m})Q)\|,\,\,m\in\bb N,\,\,1\leq i,\,j\leq n,\,\,n \in \bb N.$$
 Since $$\left\{(A_{i,j}^{n,m})\in M_{n}(\cl A)\mid A_{i,j}^{n,m}\in\cl A,\,\,m\in\bb N\right\}$$ is norm-dense in $M_n(\cl A)$ we obtain 
  $$\|(A_{i,j})\|_{M_n(\cl A)}=\|(A_{i,j}P)\|_{M_n(\cl A)},\,\,\forall\,(A_{i,j}) \in M_n(\cl A),\,n\in \bb N$$
  and
   $$\|(u(A_{i,j}))\|_{M_n(u(A))}=\|(u(A_{i,j})Q)\|_{M_n(\cl A)},\,\,\forall\, (A_{i,j}) \in M_n(\cl A),\,n\in \bb N.$$
   
  We define the map
   $w\colon \cl A P\to \cl B(L) $
as follows $$w( AP)=\begin{pmatrix} AP & 0\\ 0 & u(A)Q \end{pmatrix},\,A\in\cl A.$$
This map is bounded because for all $A\in \cl A$ we have 
\begin{align*}
    \|w(AP)\|&\le \|AP\|+\|u(A)Q\|\\&=\|A\|+\|u(A)\|\\&\le (1+\|u\|)\|A\|\\&=(1+\|u\|)\|AP\|.
\end{align*} Moreover, $w$ is a unital homomorphism and
\begin{align*}
    \|u_n\|&=\sup\{\|(u(A_{i,j}))\| \mid  (A_{i,j})\in Ball(M_{n}(\cl A))\}\\&=\sup\{\|(u(A_{i,j})Q))\| \mid (A_{i,j}P)\in Ball(M_n(\cl A P))\}\\&\le \|w_n\|.
\end{align*}
By Lemmas \ref{an} and \ref{RR} the map $w$ can be extended to a $\mathrm{w}^*$-continuous and unital bounded homomorphism 
 $$\tilde w\colon \overline{\cl A P}^{\mathrm{w}^*}\to \cl B(L).$$ Since $P(H)$ is separable, Theorem \ref{SOL} implies that the map $\tilde w$ is completely bounded and we have 
 $$\sup_{n}\|u_n\|\le \sup_{n}\|w_n\|=\sup_n\|\tilde w_n\|= \|\tilde w\|_{cb}<\infty.$$
Therefore, $u$ is completely bounded.
  \end{proof}

\begin{thm}  \label{sp} Under (CHH), all unital $C^*$-algebras satisfy (SP).
\end{thm}
\begin{proof} Let $\cl A$ be a unital $C^*$-algebra and $u\colon \cl A \to \cl B(H)$ be a unital bounded
homomorphism which is not completely bounded. Then, for each $n\in\mathbb{N}$  there
exists $a_{k_n} = (a_{i,j}^{k_n})
\in Ball(M_{k_n})(\cl A)$ such that $$||u_{k_n}||\geq n,\,\forall\,n\in\bb N,$$ where $u_{k_n}\colon M_{k_n}(\cl A)\to M_{k_n}(\cl B(H)),\,\,u_{k_n}((A_{i,j})_{i,j=1}^{k_n})=(u(A_{i,j}))_{i,j=1}^{k_n}.$
We consider the unital separable $C^*$-algebra $\cl A_0$ generated by the elements $$a_{i,j}^{k_n}\in\cl A,\,\,1\leq i,j\leq k_n,\,\,n\in\bb N.$$
By Lemma \ref{RRR} the unital bounded homomorphism $u_0=u|_{\cl A}$ is completely bounded. However, this is a contradiction since $$||(u_0)_{k_n}(a_{k_n})||\geq n,\,\forall\,n\in\bb N.$$
Therefore, every bounded homomorphism $u\colon \cl A\to \cl B(H),$ where $H$ is a Hilbert space, is completely bounded. According to \cite[Theorem 1.10]{haag}, the algebra $\cl A$ satisfies (SP).

\end{proof}

\begin{cor}
    Under (CHH), all $C^*$-algebras satisfy (SP).
\end{cor}

\begin{proof}
  According to Theorem \ref{sp}, it suffices to prove the assertion for the case of a non-unital $C^*$-algebra. Let $\cl A$ be a non-unital $C^*$-algebra and $\cl A\oplus \bb C$ be its unitization. Let $$\iota\colon \cl A\oplus \bb C\to \cl B(K)$$ be a 1-1 $*$-homomorphism, where $K$ is a Hilbert space. By Lemma 2.1.12 in \cite{bm} we have $$|\lambda|\leq ||\iota(A,0)+\lambda I_{K}||=||\iota(A,\lambda)||=||(A,\lambda)||,\,\lambda\in\bb C,\,\,A\in\cl A.$$
  Also, if $A\in\cl A$ it holds that $$||A||\leq ||(A,\lambda)||+|\lambda|\leq 2\,||(A,\lambda)||,\,\,\lambda\in\bb C.$$
  Let $u\colon \cl A\to \cl B(H),$ where $H$ is a Hilbert space, be a bounded homomorphism and define $$\tilde{u}\colon \cl A\oplus \bb C\to \cl B(H),\,\,\tilde{u}((A,\lambda))=u(A)+\lambda I_{H}.$$
  This is a well defined bounded homomorphism because \begin{align*}
      ||\tilde{u}((A,\lambda))||&=||u(A)+\lambda I_{H}||\\&\leq ||u||||A||+|\lambda|\\&\leq 2\,||u|||\,||(A,\lambda)||+||(A,\lambda)||\\&=(2\,||u||+1)\,||(A,\lambda)||,\,\,A\in\cl A,\,\,\lambda\in\bb C,
  \end{align*}
However, $\cl A\oplus \bb C$ is a unital $C^*$-algebra and Theorem \ref{sp} yields that there exists a $*$-homomorphism $\pi\colon \cl A\oplus \bb C\to \cl B(H)$ and an invertible operator $S\in\cl B(H)$ such that $$\tilde{u}(A,\lambda)=S^{-1}\pi(A,\lambda)S,\,\,A\in\cl A,\,\,\lambda\in\bb C.$$
 We set $\pi_0\colon \cl A\to \cl B(H)$ by $\pi_0(A)=\pi(A,0).$ Then $\pi_0$ is a $*$-homomorphism and $$u(A)=\tilde{u}(A,0)=S^{-1}\pi(A,0)S=S^{-1}\pi_0(A)S,\,\,A\in\cl A,$$ that is $u$ is similar to a $*$-homomorphism, as desired.
\end{proof}

Concluding this section, assuming again that {\bf (CHH)} is true we obtain the following corollaries.

\begin{cor} \label{cor1} Let $\cl A$ be a $C^*$-algebra and $\,\pi\colon \cl A\to \cl B(H)$ be a $*$-homomorphism. Then every derivation $\,\delta\colon \cl A\to \cl B(H)$ with respect to $\pi$ is inner.
\end{cor}

\begin{cor} \label{cor2} Every von Neumann algebra is hyperreflexive.

\end{cor}

\section{Injective von Neumann algebras and w*-similarity property}

In this section, we fix an injective von Neumann algebra $\cl A$ acting on the Hilbert space $H,$ as well as a von Neumann algebra $\cl B$ acting on the Hilbert space $K$ that satisfies (WSP).
We are going to prove that 
the spatial tensor product $\cl A\bar \otimes \cl B
$ satisfies (WSP). First, we need to prove some useful lemmas.

Since $\cl A^{\prime}\bar\otimes \cl B(K)$ is also an injective von Neumann algebra because it is a spatial tensor product of injective von Neumann algebras, there exists a norm one projection $$\Phi\colon \cl B(H)\bar \otimes \cl B(K)\rightarrow \cl B(H)\bar \otimes \cl B(K)$$ onto 
$ \cl A^\prime\bar \otimes \cl B(K) .$  

\begin{lem}\label{301} If $\,T\in \cl B(H)$ there exists $X\in \cl B(H)$ such that $$ \Phi (T\otimes I_K) =X\otimes I_K.$$
\end{lem}
\begin{proof} Since $ \lambda I_H\otimes M $ belongs to $ \cl A^\prime\bar \otimes \cl B(K)=Ran(\Phi) $ for all $\lambda \in \bb C,\, M\in 
\cl B(K),$ by the fact that $\Phi$ is a conditional expectation, Corollary \ref{ce}, we have that \begin{align*}
    (\lambda I_H\otimes M) \Phi ( T\otimes I_K ) &=\Phi ((\lambda I_H\otimes M)(T\otimes I_K ) ) \\&=\Phi (\lambda T\otimes M)\\&=\Phi (T\otimes I_K) (\lambda I_H\otimes M) .
\end{align*}

Thus $$\Phi ( T\otimes I_K ) \in (\bb C\,I_{H}\otimes \cl B(K))^\prime=\cl B(H)\otimes \bb C\,I_{K}$$ Therefore there exists 
$X\in \cl B(H)$ such that \[\Phi (T\otimes I_K) =X\otimes I_K\qedhere\] \end{proof}

We denote by $\cl W$ the von Neumann algebra $\cl W=\cl B(H)\bar \otimes \cl B.$

\begin{lem}\label{302} $$\Phi (\cl W)\subset \cl W.$$
\end{lem}
\begin{proof} Fix $T\in \cl B(H),\,B\in \cl B.$ Since $I_H\otimes B\in \cl A^\prime\bar \otimes \cl B(K)=Ran(\Phi)$ and $\Phi$ is a conditional expectation, we have that $$\Phi (T\otimes B)=\Phi ((T\otimes I_K)(I_H\otimes B))=
\Phi (T\otimes I_K)(I_H\otimes B).$$ By Lemma \ref{301} there exists $X\in \cl B(H)$ such that 
$\Phi (T\otimes I_K)=X\otimes I_K.$ Thus \[\Phi (T\otimes B)=(X\otimes I_K)\,(I_H\otimes B)=X\otimes B\in \cl B(H)\bar \otimes \cl B=\cl W.\qedhere\]
\end{proof}

\begin{lem}\label{303} The projection $\Phi$ maps $\cl W$ onto $\cl A^{\prime}\bar \otimes 
\cl B$, i.e. $$\Phi (\cl W)=\cl A^\prime\bar \otimes \cl B.$$
\end{lem}

\begin{proof} If $A\in \cl A^\prime$ and $B\in \cl B$ then $A\otimes B\in \cl A^\prime\bar \otimes \cl B\subseteq \cl A^\prime\bar \otimes \cl B(K)=Ran (\Phi),$ so $A\otimes B=\Phi(A\otimes B)\in \Phi(\cl A^\prime \bar \otimes \cl B)\subseteq \Phi(\cl W).$  
Thus $$\cl A^\prime\bar \otimes  \cl B\subseteq \overline{\Phi (\cl W)}^{\mathrm{w}^*}.$$ By Lemma \ref{302} we have that 
$$\Phi (\cl W)=\Phi(\cl W)\cap \cl W\subseteq (\cl A^\prime\bar \otimes \cl B(K))\cap \cl W.$$ 
We aim to show that $$(\cl A^\prime\bar \otimes \cl B(K))\cap \cl W \subset \cl A^\prime\bar \otimes \cl B.$$

For all $\omega\in (\cl A^\prime)_{*}$ we define $$R_{\omega}\colon \cl A^\prime\bar \otimes \cl B(K)\to \cl B(K),\,\,x\otimes y\mapsto \omega(x)\,y,\,\,x\in\cl A^\prime,\,y\in\cl B(K).$$
Let $$\cl F(\cl A^\prime ,\cl B)=\left\{T\in\cl A^\prime \bar \otimes \cl B(K)\mid R_{\omega}(T)\in \cl B,\,\forall\,\omega\in (\cl A^\prime)_{*}\right\}.$$
It is known that injective von Neumann algebras have the $S_{\sigma}$ property, \cite{kr}. Since $\cl A^\prime$ is injective we have that $\cl F(\cl A^\prime,\cl B)=\cl A^\prime\bar \otimes \cl B.$

Fix $T\in (\cl A^\prime\bar \otimes \cl B(K))\cap \cl W.$ 
For all $\omega \in (\cl A^\prime)_{*}$ we have that $R_{\omega}(T)\in \cl B$ and thus 
$T\in \cl F(\cl A^\prime, \cl B)=\cl A^\prime\bar\otimes \cl B.$ 
We have proved that $$ \overline{\Phi (\cl W)}^{\mathrm{w}^*} =\cl A^\prime\bar \otimes \cl B.$$
 If $X\in \cl A^\prime \bar \otimes \cl B$ then $X\in Ran \Phi \cap \cl W$ and thus $X=\Phi (X)\in \Phi (\cl W).$ 
Therefore  $\Phi (\cl W)=\cl A^\prime\bar \otimes \cl B,$ as desired.
\end{proof}

\begin{lem}\label{304} Assume that $\cl W= \cl B(H)\bar \otimes \cl B $ 
is hyperreflexive. Then $ \cl A^\prime \bar \otimes \cl B $ is hyperreflexive and $$k(\cl A^\prime \bar \otimes \cl B )\leq 
4+k(\cl W).$$
\end{lem}
\begin{proof} We recall $\Phi$ and $\cl W=\cl B(H)\bar \otimes \cl B$ from Lemma \ref{303}. %We have that $$\Phi (\cl W)=\cl A^\prime\bar \otimes \cl B.$$
If $T\in \cl B(H)\bar \otimes \cl B(K)$ then
$$d(T, \cl A^\prime\bar \otimes \cl B) =d(T, \Phi(\cl W))=  \inf_{W\in \cl W}\|T-\Phi (W)\| \leq
\|T-\Phi (T)\|+\inf_{W\in \cl W}\|\Phi (T)-\Phi (W)\|. $$ 

Since $\|\Phi \|=1$ we have that
\begin{equation}\label{ena} d(T, \cl A^\prime\bar \otimes \cl B) \leq 
\|T-\Phi (T)\|+ \inf_{W\in \cl W}\|T-W\| .  
\end{equation}

Using the same arguments as those used in the proof of \cite[Theorem 9.6]{dav}, we have
\begin{equation}
    \label{new} \|T-\Phi (T)\|\leq 4\sup\{\|P^\bot TP\|\mid P\in\Lat{ \cl A^\prime\bar\otimes \cl B(K) }\}.
\end{equation}
By Lemma \ref{4.4} it holds that $$\sup\left\{||P^{\perp} T P||\,\,\,\mid P\in\Lat{\cl A^\prime\bar\otimes \cl B(K)}\right\}=\alpha_{\cl A^\prime \bar\otimes \cl B(K)}(T)\leq r_{\cl A^\prime \bar \otimes \cl B(K)}(T)$$
and according to (\ref{new}) we get $||T-\Phi(T)||\leq 4\,r_{\cl A^\prime \bar\otimes \cl B(K)}(T).$\\
However, $r_{\cl A^\prime \bar \otimes \cl B(K)}(T)\leq r_{\cl A^\prime \bar \otimes \cl B}(T)$ since $\cl A^\prime \bar \otimes \cl B\subseteq \cl A^\prime \bar \otimes \cl B(K),$ and thus 
 \begin{equation}\label{dyo} \|T-\Phi (T)\|\leq 4r_{\cl A^\prime\bar\otimes \cl B}(T).  
\end{equation}

Furthermore, $\cl A^{\prime}\bar\otimes \cl B \subseteq \cl W\Rightarrow r_{\cl W}(T)\leq r_{\cl A^{\prime}\bar \otimes \cl B}(T),$ so  

\begin{equation}\label{tria} d(T, \cl W) \leq k(\cl W)\, r_{\cl W}(T)\leq 
k(\cl W)\, r_{\cl A^\prime\bar \otimes \cl B}(T) . \end{equation}

By inequalities (\ref{ena}), (\ref{dyo}) and (\ref{tria}) we have 
$$d(T, \cl A^\prime\bar \otimes \cl B) \leq (4+k(\cl W) )r_{\cl A^\prime\bar \otimes \cl B}(T)\Rightarrow  \dfrac{d(T, \cl A^{\prime}\bar \otimes \cl B)}{r_{\cl A^\prime\bar \otimes \cl B}(T)}\leq 4+k(\cl W).  $$
We conclude that $\cl A^\prime\bar \otimes \cl B$ is hyperreflexive and $k(\cl A^{\prime}\bar \otimes \cl B)\leq 4+k(\cl W).$
\end{proof}

We prove now the main theorem of this section. 

\begin{thm}\label{305} The algebra $\cl A\bar \otimes \cl B$ satisfies (WSP).
\end{thm}
\begin{proof} Since $\cl B$ satisfies (WSP) so does $\mathbb{C}\,I_{H}\bar\otimes \cl B$. 
%\ma{$\cl B: \text{(WSP)}$}
By Theorem \ref{101}, the algebra $$ (\bb  C\,I_{H}\bar \otimes \cl B)^{\prime}\bar \otimes \cl B(\ell^2(J))=\cl B(H) \bar \otimes 
\cl B^\prime \bar \otimes \cl B(\ell^2(J))$$ is hyperreflexive for all cardinals $J.$

By Lemma 
\ref{304}, $$\cl A^\prime \bar \otimes 
\cl B^\prime \bar \otimes \cl B(\ell^2(J))$$ is hyperreflexive for all cardinals $J.$

However 
$(\cl A\bar \otimes\cl B)^{\prime}=\cl A^\prime\bar \otimes \cl B^\prime$ and by Theorem \ref{101} we have that 
$\cl A\bar \otimes \cl B$ satisfies (WSP). 
\end{proof}

\begin{cor}\label{333} Let $\cl A$ be an injective von Neumann algebra and let $\cl B$ be a factor of type $II_1$ with property $\Gamma.$ Then $\cl A\bar \otimes \cl B$ satisfies (WSP).
\end{cor}
\begin{proof} By combining \cite[Theorem 2.3]{echris} with \cite[Theorem 1.10]{haag} it follows that the factors of type $II_1$ with property $\Gamma$ satisfy (WSP) and the conclusion follows from Theorem \ref{305}.
\end{proof}

\begin{cor}\label{FIN} Let $\cl B$ be a von Neumann algebra satisfying (WSP) and $n\in \bb N.$ Then the algebra $M_n(\cl B)$ satisfies (WSP).
\end{cor}

%We recall that a $C^*$-algebra $\cl A$ is called semi-nuclear if for every representation $\pi\colon \cl A\rightarrow \cl B(H),$ the 
%von Neumann algebra $\pi (\cl A)^{\prime \prime}$ is injective. It is known that nuclear algebras are semi-nuclear. The validitity of the converse is not yet known.

%\begin{cor}\label{310} If $\cl A$ is a  $C^*$-algebra satisfying (SP) and $\cl B$ is a semi-nuclear  $C^*$-algebra 
%then the minimal tensor product $\cl A  \otimes _{min}  \cl B$ satisfies (SP). 
%\end{cor} 
%\begin{proof} According to (reference) it holds that $$(\cl A\otimes _{min}  \cl B)^{**}= \cl A^{**}\bar \otimes \cl B^{**}\cong \cl B^{**}\bar \otimes \cl A^{**}$$ and 
%$\cl B^{**}$ is injective. 

%Furthermore, $\cl A^{**}$ satisfies (WSP) since $\cl A$ satisfies (SP), see Lemma \ref{WSP}. By Theorem \ref{305}, the second dual algebra $(\cl A\otimes _{min}  \cl B)^{**}$ satisfies (WSP). Again, by Lemma \ref{WSP},
%we deduce that the $C^*$-algebra $\cl A\otimes _{min}  \cl B$ satisfies (SP).
%\end{proof}

%Some other results on the problem of determining when the tensor product of two $C^*$-algebras 
%satisfies (SP) exist in \cite{fp}.

In the same context, i.e. for an injective von Neumann algebra $\cl A$ acting on the Hilbert space $H$ and a von Neumann algebra $\cl B$ acting on the Hilbert space $K$ satisfying (WSP), we are going to prove that the w*-similarity degree of $\cl A\bar \otimes \cl B$ is less than $2k_0+8,$ where $$k_0=\sup_{J}k(\cl B^\prime \bar \otimes \cl B(\ell^2(J)))$$ and the above supremum is taken over all infinite cardinals $J.$ Since $\cl B$ satisfies (WSP), by Remark \ref{ttt}, we have that $k_0<\infty.$ Before we prove this result (see Theorem \ref{402}) we recall the definition of the w*-similarity degree.

Let $\cl C$ 
be a von Neumann algebra satisfying (WSP). In \cite[Theorem 4.1]{lemerd}, Christian Le Merdy  proved that there exists a constant $C>0$ 
and a positive integer $N$ such that for any w*-continuous bounded unital homomorphism 
$u\colon \cl C \rightarrow \cl B(L)$ ($L$: Hilbert space) it  holds that 
\begin{equation}\label{xxxv} \|u\|_{cb}\leq C\|u\|^N.
\end{equation}
The least positive integer $N$ satisfying (\ref{xxxv}) is called w*-similarity degree of $\cl C$ and it is denoted by $d_{*}(\cl C).$

\begin{thm}\label{402} Let $\cl A$ be an injective von Neumann algebra acting on the Hilbert space $H$ and $\cl B$ be a von Neumann algebra acting on the Hilbert space $K$ satisfying (WSP). Then  $$d_{*}( \cl A\bar \otimes \cl B )\leq 2k_0+8.$$
\end{thm}
\begin{proof} By Theorem \ref{305}, the algebra $\cl C=\cl A\bar \otimes \cl B $ satisfies (WSP). In Lemma \ref{12} we proved that for every inner derivation $\delta=\delta_{\pi,\,T}$ holds that $||\delta||_{cb}\leq 2 \theta ||\delta||,$ where $$\theta = \sup_J(k((\cl A\bar \otimes \cl B)^{\prime}\bar \otimes \cl B(\ell^2(J))))= \sup_J(k(\cl A^\prime \bar \otimes \cl B^\prime \bar \otimes \cl B(\ell^2(J)))),$$ $\,T\in\cl B(L)$ and $\pi$ is a normal representation of $\cl A\bar\otimes \cl B$ into $\cl B(L),$ where $L$ is a Hilbert space. 
According to \cite[Remark 10.8]{pisier1}, we have 
$\|u\|_{cb}\leq \|u\|^{2\theta }$ for every w*-continuous, unital and bounded homomorphism 
$u\colon \cl C \rightarrow \cl B(L),$ where $L$ is a Hilbert space.
However, the algebras $\cl B(H)\bar \otimes \cl B^\prime \bar \otimes \cl B(\ell^2(J))$ and $\cl B^\prime\bar \otimes \cl B(H\otimes \ell^2(J))$ are unitarily equivalent and since $\cl B^\prime \bar \otimes \cl B(H\otimes \ell^2(J))$ is hyperreflexive, the algebra $\cl B(H)\bar \otimes \cl B^\prime \bar \otimes \cl B(\ell^2(J))$ is also hyperreflexive. By Lemma \ref{304}, it follows that  $\theta \leq 4+k_0$ 
and thus $||u||_{cb}\leq ||u||^{2\theta}\leq ||u||^{2k_0+8}$ for every w*-continuous bounded unital homomorphism 
$u\colon \cl C \rightarrow \cl B(L),$ where $L$ is a Hilbert space. Therefore, $d_*(\cl A\bar \otimes \cl B)\leq 2k_0+8.$
\end{proof}

\textbf{Acknowledgements}: We are grateful to Professors M. Anoussis and A. Katavolos for their helpful comments and suggestions during the preperation of this paper. We would like also to express our appreciation to Professor N. Ozawa for providing us useful information about type $II_1$ factors. Finally, we would like to thank the referee for carefully reading the paper and making significant improvements.

------------------------------------------------------------------------
\end{document}